\documentclass[12pt]{article}
\usepackage{amsmath,amsthm,amsfonts,amssymb,amscd}
\usepackage[usenames]{color}
\usepackage[dvips]{graphicx}
\usepackage{epsfig}
\def\phi{\varphi}
\def \P {\mathbb P}

\usepackage{multirow}
\begin{document}

\newtheorem{lem}{Lemma}
\newtheorem{predl}{Proposition}
\newtheorem{theorem}{Theorem}
\newtheorem{defin}{Definition}
\newtheorem{zam}{Remark}

\bigskip
\centerline{\bf ON ASYMPTOTIC EFFICIENCY OF GOODNESS-OF-FIT TESTS }
\centerline{\bf FOR THE PARETO DISTRIBUTION}
\centerline{\bf BASED ON ITS CHARACTERIZATION}

\bigskip

\centerline{Volkova K. Yu.\footnote{Research supported by grant RFBR No. 13-01-00172, grant NSh No. 2504.2014.1 and by SPbGU grant No. 6.38.672.2013}}

\bigskip
\centerline{Saint-Petersburg State University, Russia}

\begin{abstract}
We introduce a new characterization of Pareto distribution and construct integral
and supremum type goodness-of-fit tests based on it. Limiting distribution and large deviations
of new statistics are described and their local Bahadur efficiency for
parametric alternatives is calculated. Conditions of local optimality of new
statistics are given.
\end{abstract}



 \section{ Introduction}
\bigskip

Let ${\cal P}$ be the family of Pareto distributions with the distribution function (d.f.)
\begin{equation}
\label{DF}
F(x) = 1-x^{-\lambda}, \ x \geq 1, \  \lambda>0.
\end{equation}

In this paper we develop the goodness-of-fit tests for Pareto distribution using a new characterization based on the property of order statistics.
The problem formulation is as follows: let $X_1,\ldots,X_n$ be positive i.i.d. rv's with continuous d.f. $F.$
Consider testing the composite hypothesis $H_0: F \in {\cal P}$ against the general
alternative $H_1: F  \notin {\cal P},$ assuming that the alternative d.f. is also concentrated on $[1,\infty).$

It is well known that the log-transform of a Pareto random variable has an
exponential distribution. Therefore the tests of exponentiality are used by many authors to test the {\it Paretianity} of the sample.
Our approach for this problem is unlike and uses  directly  the initial Pareto sample.

The goodness-of-fit tests  for the Pareto distribution have been discussed in \cite{BWG}, \cite{Gul}, \cite{Mart}, \cite{Riz}. We
exploit the different idea for constructing and analyzing statistical tests based on characterization by the property of equidistribution of linear statistics by means of so-called $U$-empirical d.f.'s, see \cite{Jan}, \cite{Kor}.
This method was developed early in several articles, particularly, in \cite{Niknik}, \cite{NiPe}, \cite{NikTchir}, \cite{NikVol}, \cite{NikVol_PR}, \cite{Litv2}.
The tests for the Pareto distribution using this approach were obtained and analyzed in \cite{Boj}.
One can observe that the new tests based on characterizations have reasonably high efficiencies and can be competitive with previously known goodness-of-fit tests. Let us explain our approach.

We will say that the d.f. $F$ belongs to the class of distributions $\cal F$,
if  $\forall x_1, x_2:$ either $F(x_1x_2)\leq F(x_1)F(x_2)$ or $F(x_1x_2)\geq F(x_1)F(x_2),$ see
\cite{Ah1}.

Let $X_1,\ldots,X_n$ be i.i.d. positive absolutely continuous random variables with d.f. $F$ from
the class $\cal F.$
Denote by $X_{(1,n)}\leq X_{(2,n)}\leq\ldots \leq X_{(n,n)}$ - the order statistics of a random sample $X_1,...,X_n.$

We present a new characterization within the class $\cal F$.
\begin{theorem} Let $X_1,...,X_k$ be i.i.d., positive and bounded random variable having an absolutely continuous (with respect to Lebesgue measure) d.f. $F(x).$
Then the equality in law of $X_1$ and $X_{(k,k)}/X_{(k-1,k)}$ takes place iff  $X_1$ has some d.f. from the family $\cal P$.
\end{theorem}

\begin{proof}
Let $Y=\ln{X}$ and let $G$ denote the d.f. of $Y.$ It can be easily seen that $F \in \cal F$ iff $G$ is NBU ("new better than used") or NWU ("new worse than used") (see \cite{Ah2}).
Further, since we use the monotonic transformation, then $X_1$ and $X_{(k,k)}/X_{(k-1,k)}$ will be identically distributed iff $Y_1$ and $Y_{(k,k)}-Y_{(k-1,k)}$ are identically distributed. It follows from \cite{Ah2} that
$X_1$ and $X_{(k,k)}/X_{(k-1,k)}$ are identically distributed iff $Y=\ln{X}$ has the exponential distribution with some scale parameter $\lambda,$ therefore $X_1$ has the Pareto distribution with the same parameter $\lambda.$
\end{proof}

In the case when $k=2$ our characterization coincide with another characterization of Pareto distribution considered in \cite{Boj}, see also \cite{NikVol_PR}.
Note that our characterization extend the charaterization, involved in \cite{Boj}.

According to our characterization we construct
the $U$-empirical d.f. by the formulae
\begin{gather*}
H_n(t)={n \choose
k}^{-1}\sum_{1 \leq i_1<\ldots < i_k  \leq n}\textbf{1}\{X_{(k,\{i_1,\ldots,i_k\})}/X_{(k-1,\{i_1,\ldots,i_k\})}< t\}, \quad t\geq 1,
\end{gather*}
where  $X_{(s,\{i_1,\ldots,i_k\})}, \, s=\{k-1,k\}$ \ denotes the $s-$th
order statistic of the subsample $X_{i_1},\ldots,X_{i_k}$. For rv $X$ the $U-$statistical d.f. will be simply the usual empirical
d.f. $F_n(t)=n^{-1}\sum_{i=1}^n\textbf{1}(X_i<t), t \in R^1,$
based on the observations $X_1,\dots,X_n.$

It is known that the properties of $U$-empirical d.f.'s are similar to
the properties of usual empirical d.f.'s, see \cite{HJS}, \cite{Jan}. Hence the difference $H_n- F_n$ for large $n$  should be almost surely  close to zero
 under $H_0,$ and we can measure their closeness by using some test statistics, assuming their large values to be critical.

We suggest two test statistics
\begin{align}
I_n^{(k)}&=\int_{1}^{\infty} \left(H_n(t)-F_n(t)\right)dF_n(t),\label{I_n}\\
D_n^{(k)}&=\sup_{t \geq 1}\mid H_n(t)-F_n(t)\mid.\label{D_n}
\end{align}
Note that both proposed statistics under $H_0$ are invariant with respect to the change of variables
$ X \to X^{\frac{1}{\lambda}},$ so we may set $\lambda =1.$

We discuss their limiting distributions under the null hypothesis and find logarithmic asymptotics of large deviations under $H_0.$  Next we calculate
their efficiencies against some parametric alternatives from the class $\cal F.$

Finally, we study the conditions of local optimality of our tests and describe the
"most favorable" \, alternatives for them.

\bigskip

 \section{ Integral statistic  $I_{n}^{(k)}$}

\bigskip

The statistic $I_{n}^{(k)}$ is asymptotically equivalent to the $U$-statistic of degree  $(k+1)$  with
the centered kernel
\begin{gather*}
\Psi_k(X_{i_1},\ldots, X_{i_{k+1}})=
\frac{1}{k+1}\sum_{\pi(i_1, \ldots,
i_{k+1})}\textbf{1}(X_{(k,\{i_1,\ldots,i_k\})}/X_{(k-1,\{i_1,\ldots,i_k\})}  < X_{i_{k+1}})
 -\frac12,
\end{gather*}
where $\pi(i_1, \ldots, i_{k+1})$ means all permutations of different indices  from $\{i_1, \ldots,i_{k+1}\}.$

Let $X_1,\ldots, X_{k+1}$ be independent rv's from standard Pareto distribution. It is known that non-degenerate $U$-statistics
are asymptotically normal, see \cite{Hoeffding}, \cite{Kor}. To prove that
 the kernel $\Psi_k(X_{1},\ldots, X_{{k+1}})$ is non-degenerate, we calculate its projection $\psi_k(s).$ For a fixed $X_{{k+1}}=s, \, s\geq 1$ we have:
\begin{multline*}
\psi_k(s) := E(\Psi_k(X_{1},\ldots, X_{{k+1}})\mid X_{{k+1}}=s) =\\
=\frac{k}{k+1}\P(X_{(k,\{2,\ldots, k, s\})}/X_{(k-1,\{2,\ldots, k, s\})} < X_{1})
+\frac{1}{k+1}\P(X_{(k,\{1,\ldots,k\})}/X_{(k-1,\{1,\ldots, k\})}< s)-\frac12.
\end{multline*}

It follows from the above characterization that the second probability is equal to:
$$
\P(X_{k,\{1,\ldots, k\}}/X_{k-1,\{1,\ldots, k\}}< s)=\P (X_1<s) = F(s).
$$

It remains to calculate the first term. For this purpose we decompose the probability as
$\P(X_{k,\{2,\ldots, k,s \}}/X_{k-1,\{2,\ldots, k,s\}} < X_{1}) = \P_1+\P_2+\P_3,$
where $\P_i, i={1,2,3}$ are initial probabilities, computed in one of the following cases:
\begin{itemize}
\item[(1)] Let the sample units take places as follows: $ X_2 < \ldots < X_k <s.$ Then our probability transforms into
\begin{multline*}
\P_1=(k-1)! \, \P(\frac{s}{X_k} <X_1, X_2 < \ldots < X_k <s)=\\
=(k-1)! \, \P(X_k< s, X_1 > \frac{s}{X_k} , X_2 < X_3, X_3 <X_4, \ldots, X_{k-1} < X_{k}).
\end{multline*}

After some calculations we obtain that the last probability is equal to:
\begin{multline*}
(k-1)! \int_1^s (1-F(\frac{s}{x_k}))\frac{F^{k-2}(x_k)}{(k-2)!} d F(x_k) =
F^{k-1}(s)-(k-1) \int_1^s (1-\frac{1}{x})^{k-2} (1-\frac{x}{s}) \frac{dx}{x^2}.
\end{multline*}

The integral in the second term can be evaluated using integration by parts and binomial representation of the function $(1-\frac1x)^{k-1}.$ Finally we have:
\begin{multline*}
\int_1^s (1-\frac{1}{x})^{k-2} (1-\frac{x}{s}) \frac{dx}{x^2} = \frac{1}{s(k-1)}\int_{1}^{s}\sum_{j=0}^{k-1}(-1)^j {{k-1} \choose j}x^{-j} dx =\\
= \frac{1}{s(k-1)}\left(s-1-(k-1)\ln{(s)}+\sum_{j=2}^{k-1}(-1)^j {{k-1} \choose j}\frac{1-s^{-(j-1)}}{j-1}\right).
\end{multline*}

Thus the initial probability in this case is equal to
\begin{gather*}
\P_1=F^{k-1}(s)-F(s)+(k-1)\frac{\ln{s}}{s}-\frac{1}{s}\sum_{j=2}^{k-1}(-1)^j {{k-1} \choose j}\frac{1-s^{-(j-1)}}{j-1}.
\end{gather*}

\item[(2)] The sample units are $X_2 < X_3<\ldots X_{k-1}< s < X_k,$ then for this case we have:
\begin{align*}
\P_2=&(k-1)! \, \P(\frac{X_k}{s} <X_1, X_2 < X_3<\ldots X_{k-1}< s < X_k)=\\
=&(k-1)! \, \P(X_k> s, X_1 > \frac{X_k}{s} , X_2 < X_3, X_3 < X_4, \ldots, X_{k-1}<s)=\\
=&(k-1)! \int_s^{\infty} (1-F(\frac{x_k}{s}))\frac{F^{k-2}(s)}{(k-2)!} d F(x_k) =\\
=&\frac{(k-1)}{2s}F^{k-2}(s).
\end{align*}

\item[(3)] The last case we consider is when $s$ is situated on $j-$th place ($1 \leq j \leq {k-2}$) in variational series of the sample
$X_2, \ldots, X_{k-2}.$ It means that the sample units take places as follows: $ X_2<  \ldots < s < \ldots <X_{k-2}< X_{k-1} < X_k$
and $s$ also may stand on first and $(k-2)$-th places. Then the required probability is equal to
\begin{multline*}
\P_3=(k-1)! \, \P(\frac{X_k}{X_{k-1}} <X_1, X_2<  \ldots < s < \ldots <X_{k-2}< X_{k-1} < X_k)=\\
=\frac12 C_{k-1}^{j-1}(1-F(s))^{k-j} F^{j-1}(s), \, 1 \leq j \leq {k-2}.
\end{multline*}

\end{itemize}

Combining the results we get that the first term in the projection has the form:
\begin{multline*}
\P(X_{(k,\{ 2,\ldots, k, s\})}/X_{(k-1,\{2,\ldots, k, s\})} < X_{1}) =
F^{k-1}(s)-F(s)+(k-1)\frac{\ln{s}}{s}-\\-\frac{1}{s}\sum_{j=2}^{k-1}(-1)^j {{k-1} \choose j}\frac{1-s^{-(j-1)}}{j-1}+\frac12 \sum_{j=1}^{k-1}C_{k-1}^{j-1}(1-F(s))^{k-j} F^{j-1}(s).
\end{multline*}

Note that the last sum is equal to $\sum_{j=1}^{k-1}C_{k-1}^{j-1}(1-F(s))^{k-j} F^{j-1}(s)=1-F^{k-1}(s).$
Thus for the initial probability we get the result:
\begin{multline*}
\P(X_{(k,\{ 2,\ldots, k, s\})}/X_{(k-1,\{2,\ldots, k, s\})} < X_{1}) =
\frac12 F^{k-1}(s) -\\-F(s)+(k-1)\frac{\ln{s}}{s}-\frac{1}{s}\sum_{j=2}^{k-1}(-1)^j {{k-1} \choose j}\frac{1-s^{-(j-1)}}{j-1}+\frac12.
\end{multline*}

Hence we get the final expression for the projection of the kernel $\Psi_k:$
\begin{align}
\psi_k(s) =\frac{kF^{k-1}(s)-1}{2(k+1)}-&\frac{k-1}{k+1}F(s)+\frac{k(k-1)}{k+1}\frac{\ln{s}}{s} - \notag
\\ - &\frac{k}{s(k+1)} \sum_{j=2}^{k-1}(-1)^j {{k-1} \choose j}\frac{1-s^{-(j-1)}}{j-1}.\label{psi_k}
\end{align}

The calculation of this variance for the projection $\psi_k$ in the general case  is too complicated, therefore we calculate it only for particular $k$.

\subsection{ Integral statistic $I_{n}^{(3)}$}

The projection $\psi_k(s)$ for case $k=3$ has the form:
\begin{equation}
\label{psi_3}
\psi_3(s)=\frac{9}{8s^2}+\frac{3\ln{s}}{2s}-\frac1s-\frac14.
\end{equation}

The variance of this projection $\Delta_3^2 = E\psi_3^2(X_1)$ under $H_{0}$
is  given by
\begin{gather*}
\Delta_3^2 = \int_{1}^{\infty} \psi_3^2 (s) \frac{1}{s^2}ds =\frac{11}{1920}
\approx 0.0057.
\end{gather*}

Therefore the kernel $\Psi_3$ is centered and non-degenerate.
We can apply Hoeffding's theorem on asymptotic normality of $U$-statistics, see again \cite{Hoeffding}, \cite{Kor},
which implies that the following result holds
\begin{theorem}
Under null hypothesis as $n \rightarrow \infty$ the statistic $\sqrt{n}I_{n}^{(3)}$ is asymptotically normal
so that
$$\sqrt{n}I_{n}^{(3)} \stackrel{d}{\longrightarrow}{\cal
{N}}(0,\frac{11}{120}).$$
\end{theorem}

Now we shall evaluate the large deviation asymptotics of the sequence of statistics $I_{n}^{(3)}$ under $H_0.$
According to the theorem on large deviations of such statistics from \cite{nikiponi}, see also \cite{anirban},
 \cite{Niki10}, we obtain due the fact that the kernel $\Psi_3$ is centered, bounded and non-degenerate
the following result.
\begin{theorem}
For  $a>0$
$$
\lim_{n\to \infty} n^{-1} \ln P ( I_n^{(3)} >a) = - f_I^{(3)}(a),
$$
where the function $f_I^{(3)}$ is continuous for sufficiently small $a>0,$ and $$
f_I^{(3)}(a)  \sim \frac{a^2}{32 \Delta_3^2} = 5.455\, a^2, \, \mbox{ as } \, a \to 0.
$$
\end{theorem}

\subsection{Some notions from Bahadur theory}

Suppose that under the alternative $H_1$ the observations have the d.f. $G(\cdot,\theta)$ and the density $g(\cdot,\theta), \ \theta \geq 0,$ such that
$G(\cdot, 0) \in {\cal P}.$
The measure of Bahadur efficiency (BE) for any sequence $\{T_n\}$ of test statistics is the exact slope
$c_{T}(\theta)$ describing the rate of exponential decrease for the
attained level under the alternative d.f. $G(\cdot,\theta).$ According to Bahadur theory  \cite{Bahadur}, \cite{Nik} the exact slopes may be found by
using the following Proposition.

\noindent {\bf Proposition}.\,{\it Suppose that the following two
conditions hold:
\[
\hspace*{-3.5cm} \mbox{a)}\qquad  T_n \
\stackrel{\mbox{\scriptsize $P_\theta$}}{\longrightarrow} \
b(\theta),\qquad \theta > 0,\nonumber \] where $-\infty <
b(\theta) < \infty$, and $\stackrel{\mbox{\scriptsize
$P_\theta$}}{\longrightarrow}$ denotes convergence in probability
under $G(\cdot\ ; \theta)$.
\[
\hspace*{-2cm} \mbox{b)} \qquad \lim_{n\to\infty} n^{-1} \ \ln \
P_{H_0} \left( T_n \ge t \ \right) \ = \ - h(t)\nonumber
\] for any $t$ in an open interval $I,$ on which $h$ is
continuous and $\{b(\theta), \: \theta > 0\}\subset I$. Then
$$c_T(\theta) \ = \ 2 \ h(b(\theta)).$$}

We have already found the large deviation asymptotics. In order to evaluate the exact slope it remains to calculate the first condition of the Proposition.

Note that the exact slopes for any $\theta$ satisfy the inequality (see \cite{Bahadur}, \cite{Nik})
\begin{equation}
\label{Ragav}
c_T(\theta) \leq 2 K(\theta),
\end{equation}
where $K(\theta)$ is the Kullback-Leibler "distance"\, between the alternative and the null-hypothesis $H_0.$
In our case $H_0$ is composite, hence for any alternative density $g_j(x,\theta)$ one has
$$
K_j(\theta) = \inf_{\lambda>0} \int_1^{\infty} \ln [g_j(x,\theta) / \lambda x^{-\lambda-1} ] g_j(x,\theta) \ dx.
$$
This quantity can be easily calculated as $\theta \to 0$ for particular alternatives.
According to (\ref{Ragav}), the local BE of the sequence of statistics ${T_n}$ is defined as
$$
e^B (T) = \lim_{\theta \to 0} \frac{c_T(\theta)}{2K(\theta)}.
$$

\subsection{Local Bahadur efficiency of $I_n^{(3)}$}
According to Bahadur theory, the considered alternatives should be close to null-hypothesis as $\theta \to 0.$ Therefore we suggest three alternatives against Pareto
distribution. The first two alternatives we consider are obtained by skewing mechanism, see \cite{LP}, we call them Ley-Paindaveine alternatives.
\begin{enumerate}

\item[i)] First Ley-Paindaveine alternative with the d.f.
$$ G_1(x,\theta)=F(x)e^{-\theta(1-F(x))},\theta \geq 0, x \geq 1;$$

\item[ii)] Second Ley-Paindaveine alternative with the d.f.
$$ G_2(x,\theta)=F(x)-\theta\sin{\pi F(x)}, \theta \in [0,\pi^{-1}], x\geq 1;$$

\item[iii)] log-Weibull alternative with the d.f.
$$ G_3(x,\theta)=1-e^{-(\ln{x})^{\theta+1}},\theta \in (0,1), x\geq 1.$$
\end{enumerate}

Let us find the local BE for alternative under consideration.

According to the Law of Large Numbers
for $U$-statistics \cite{Kor}, the limit in probability under $H_1$ is equal to
\begin{gather*}
b_1(\theta)=P_{\theta}(X_{(3,3)}/X_{(2,3)}<Y)-\frac12.
\end{gather*}
It is easy to show (see also \cite{Boj}) that
$$b_1(\theta) \sim 4\theta  \int_{1}^{\infty} \psi_3(s)h_1(s)ds,
$$
where $h_1(s)=\frac{\partial}{\partial\theta}g_1(s,\theta)\mid _{\theta=0}$ and $\psi_3(s)$ is the projection from \eqref{psi_3}.
Therefore for the first Ley-Paindaveine alternative we have
\begin{gather*}
b_1(\theta) \sim 4\theta
\int_{1}^{\infty}(\frac{9}{8s^2}+\frac{3\ln{s}}{2s}-\frac1s-\frac14)
(\frac{s-2}{s^3})\frac{ds}{s^2} \sim  \frac{\theta}{12}, \quad \theta \to 0,
\end{gather*}
and the local exact slope of the sequence $I_n^{(3)}$ as $\theta \to 0$ admits the representation
$$c_1(\theta)=b^2_1(\theta)/(16\Delta_3^2) \sim \frac{5}{66}\,\theta^2, \, \theta
\to 0.$$

The Kullback-Leibler "distance" \, $K_1(\theta)$ between the alternative and the null-hypothesis $H_0$ admits
the following asymptotics (see again \cite{Boj}):
\begin{gather*}
2K_1(\theta)\sim \theta^2 \left[\{\int_1^\infty h_1^2(x)x dx -(\int_1^\infty h_1(x) \ln{(x)}dx)^2\right] ,\, \theta \to 0.
\end{gather*}
Therefore in our case
\begin{equation}\label{K1}
K_1(\theta) \sim \theta^2/24, \,\theta \to 0.
\end{equation}
Consequently, the local efficiency of the test is
$$e^B_1(I)=\lim_{\theta \to 0}\frac{c_1(\theta)}{2K_1(\theta)}\approx  \frac{10}{11}\approx 0.909.$$

Omitting the calculations  similar to previous cases, we get for the second Ley-Paindaveine alternative $
b_2(\theta)\sim 0.353\, \theta,$ $c_2(\theta)\sim 1.363\,\theta^2,$ $ \theta \to 0.$ It is easy to show that
$K_2(\theta) \sim 0.753\, \theta^2,$ $\theta \to 0.$ Therefore the local BE is equal to 0.905.

After some calculations in case of the log-Weibull alternative we have:
\begin{gather*}
b_3(\theta) \sim (\frac{3}{4}-\ln{3}+\ln{2})\theta \approx 0.345\, \theta, \quad \theta \to 0,
\end{gather*}
and the local exact slope of the sequence $I_n$ as $\theta \to 0$ admits the representation
$c_3(\theta) \sim 1.295\, \theta^2.$
Moreover for the log-Weibull distribution $K_3(\theta)$ satisfies
$K_3(\theta) \sim \frac{\theta^2}{12},$ $\theta \to 0.$ Hence the local BE for the last case is equal to $ 0.787.$

Next table \ref{fig: tab1} gathers the values of local BE.
\begin{table}[!hhh]\centering
\caption{Local Bahadur efficiency for $I_n^{(3)}$ }
\bigskip
\begin{tabular}{|c|c|}
\hline
Alternative & Efficiency\\
\hline
Ley-Paindaveine 1 &  0.909  \\
Ley-Paindaveine 2 &   0.905 \\
log-Weibull&  0.787  \\
\hline
\end{tabular}
\label{fig: tab1}
\end{table}
\bigskip

\subsection{ Integral statistic $I_{n}^{(4)}$}

For case $k=4$ the projection $\psi_k(s)$ has the form:
\begin{equation}
\label{psi_4}
\psi_4(s)=\frac{12\ln{s}}{5s}-\frac{4}{5s^3}+\frac{18}{5s^2}-\frac{13}{5s}-\frac{3}{10}.
\end{equation}

The variance of this projection under $H_{0}$ is equal to
\begin{gather*}
\Delta_4^2 = \int_{1}^{\infty} \psi_4^2 (s) \frac{1}{s^2} d s =\frac{271}{52500}
\approx 0.00516.
\end{gather*}

Therefore the kernel $\Psi_4$ is centered, non-degenerate and bounded.
Due to Hoeffding's theorem on asymptotic normality of $U$-statistics, see again \cite{Hoeffding}, \cite{Kor},
we have that:
\begin{theorem}
Under null hypothesis as $n \rightarrow \infty$ the statistic $\sqrt{n}I_{n}^{(4)}$ is asymptotically normal
so that
$$\sqrt{n}I_{n}^{(4)} \stackrel{d}{\longrightarrow}{\cal
{N}}(0,\frac{271}{2100}).$$
\end{theorem}

The large deviation asymptotics of the sequence of statistics $I_{n}^{(4)}$ under $H_0$ follows from the following
result. It was derived using the theorem on large deviations (see again \cite{nikiponi}, \cite{anirban},
 \cite{Niki10}), applied to the centered, bounded and non-degenerate kernel $\Psi_4.$
\begin{theorem}
For  $a>0$
$$
\lim_{n\to \infty} n^{-1} \ln P ( I_n^{(4)} >a) = - f_I^{(4)}(a),
$$
where the function $f_I^{(4)}$ is continuous for sufficiently small $a>0,$ and $$
f_I^{(4)}(a)  \sim \frac{a^2}{50 \Delta^2_4} = 3.875\, a^2, \, \mbox{as } \, a \to 0.
$$
\end{theorem}

\subsection{Local Bahadur efficiency of $I_n^{(4)}$}

For this case  the limit in probability under $H_1$ has the following asymptotics
$$b_1(\theta) \sim 5\theta  \int_{1}^{\infty} \psi_4(s)h_1(s)ds,
$$
where again $h_1(s)=\frac{\partial}{\partial\theta}g_1(s,\theta)\mid _{\theta=0}$ and $\psi_4(s)$ is the projection from \eqref{psi_4}.
Therefore for the first Ley-Paindaveine alternative we have
\begin{gather*}
b_1(\theta) \sim 5\theta
\int_{1}^{\infty}(\frac{9}{8s^2}+\frac{3\ln{s}}{2s}-\frac1s-\frac14)
(\frac{s-2}{s^3})\frac{ds}{s^2} \sim  \frac{\theta}{12}, \quad \theta \to 0.
\end{gather*}
and the local exact slope of the sequence $I_n^{(4)}$ as $\theta \to 0$ admits the representation
$$c_1(\theta)=b^2_1(\theta)/(25\Delta_4^2) \sim \frac{5}{66}\,\theta^2,\theta
\to 0.$$

The Kullback-Leibler "distance" for this alternative was already found above, and it satisfies
$K_1(\theta) \sim \theta^2/24, \,\theta \to 0.$ Thus the local efficiency of the test is
$$e^B_1(I)=\lim_{\theta \to 0}\frac{c_1(\theta)}{2K_1(\theta)}\approx   0.930.$$

For other alternatives the calculations are similar. Omitting the details, let us gather the values of local BE for this case in the table \ref{fig: tab2}.
\begin{table}[!hhh]\centering
\caption{Local Bahadur efficiency for $I_n^{(4)}$ }
\bigskip
\begin{tabular}{|c|c|}
\hline
Alternative & Efficiency\\
\hline
Ley-Paindaveine 1 &  0.930  \\
Ley-Paindaveine 2 &   0.961 \\
log-Weibull&  0.746  \\
\hline
\end{tabular}
\label{fig: tab2}
\end{table}

In table \ref{fig: tab3} we present the efficiencies from  tables \ref{fig: tab1} and \ref{fig: tab2} gathered with maximal values
of efficiencies against presumed alternatives.
\begin{table}[!hhh]\centering
\bigskip
\caption{Comparative table of local efficiencies for statistic $I_n^{(k)}$ }
\bigskip
\begin{tabular}{|c|c|c|c|}
\hline
 & \multicolumn{3}{|c|}{Efficiency}\\
\cline{2-4}
\raisebox{1.5ex}[0cm][0cm]{Alternative}
 & $k=3$ & $k=4$ & $\max_k$ \\
\hline
Ley-Paindaveine 1 & 0.909 &  0.930 &  0.930 for $k=4$ \\
Ley-Paindaveine 2 & 0.905 &  0.961 &  0.961 for $k=4$\\
log-Weibull& 0.787 &  0.746  &  0.821 for $k=2$\\
\hline
\end{tabular}
\label{fig: tab3}
\end{table}

\section{Kolmogorov-type statistic $D_n^{(k)}$}

\bigskip

Now we consider the Kolmogorov type statistic (\ref{D_n}).
For fixed  $t$ the difference $H_n(t) - F_n(t)$
is a family of $U$-statistics with the kernels, depending on $t\geq 1:$
\begin{gather*}
\Xi_k(X_{i_1},\ldots, X_{i_k};t)= \textbf{1}(X_{(k,\{i_1,\ldots,i_k\})}/X_{(k-1,\{i_1,\ldots,i_k\})}  < t)
-\frac{1}{k}\sum_{l=1}^{k} \textbf{1}(X_l <t) .
\end{gather*}

The projection of this kernel $\xi_k(s;t)$ for fixed
$t \geq 1$ has the form:
\begin{gather*}
\xi_k(s;t) := E(\Xi_k(X_{1},\ldots, X_{k})\mid X_{k}=s)=\\=
\P(X_{(k,\{1,\ldots,k-1, s\})}/X_{(k-1,\{1,\ldots,k-1, s\})}  < t)-\frac{1}{k} \textbf{1}\{s
<t\}-\frac{k-1}{k}\P\{X_1 <t\}.
\end{gather*}

It remains to calculate the first term. For this purpose like in the previous cases, we write the de\-composition
$$\P(X_{k,\{1, \ldots, k-1, s\}}/X_{k-1,\{1, \ldots, k-1, s\}} < t) = \P_1+\P_2+\P_3,$$ where
$\P_i, i=1,2,3,$ are the initial probabilities, computed in one of the following cases:
\begin{itemize}
\item[(1)] Let the sample units take places as follows: $ X_1 < X_2< \ldots < X_{k-1} <s.$ Then the probability expresses as
\begin{multline*}
\P_1=(k-1)! \, \P(\frac{s}{X_{k-1}} < t, X_1 < X_2< \ldots < X_{k-1} < s)=\\
=(k-1)! \, \textbf{1}( s\geq t) \P(\frac{s}{t}<X_{k-1}< s, X_1 < X_2< \ldots < X_{k-1})+\\
+(k-1)! \, \textbf{1}( s < t) \P(X_1 < X_2< \ldots < X_{k-1} < s)=\\=
\textbf{1}( s\geq t) (F^{k-1}(s)-F^{k-1}(\frac{s}{t})).
\end{multline*}

\item[(2)] The sample units are $X_1 < X_2<\ldots X_{k-2}< s < X_{k-1},$ then for this case we have:
\begin{multline*}
\P_2=(k-1)! \, \P(\frac{X_{k-1}}{s} < t, X_1 < X_2<\ldots X_{k-2}< s < X_{k-1})=\\
=(k-1)! \, \P(s < X_{k-1} < st, X_1 < X_2<\ldots X_{k-2}< s)=\\=
(k-1)! \frac{F^{k-2}(s)}{(k-2)!}(F(st)-F(s))
=\frac{(k-1)}{s}(1-\frac{1}{s})^{k-2}(1-\frac{1}{t}).
\end{multline*}

\item[(3)] In the last case let $s$ be situated on $l-$th place ($1 \leq l \leq {k-2}$) in the variational series of the sample
$X_1, \ldots, X_{k-2}$. Then the required probability transforms into:
\begin{gather*}
\P_3=(k-1)! \, \P(\frac{X_{k-1}}{X_{k-2}} < t, X_1 < \ldots <s< \ldots <  X_{k-2} < X_{k-1})=\\
= (1-\frac{1}{t})C_{k-1}^{l-1}(1-F(s))^{k-j} F^{j-1}(s), \, 1 \leq l \leq {k-2} .
\end{gather*}

\end{itemize}

Combining these results we get that the first term in the projection is equal to:
\begin{multline*}
\P(X_{(k,\{1, \ldots, k-1, s\})}/X_{(k-1,\{1, \ldots, k-1, s\})} < t) =\\=
 \textbf{1}( s\geq t) (F^{k-1}(s)-F^{k-1}(\frac{s}{t}))+
 (1-\frac{1}{t})\sum_{l=1}^{k-1}C_{k-1}^{l-1}(1-F(s))^{k-j} F^{j-1}(s).
\end{multline*}

Again we can see that the last sum can be simplified as
$$\sum_{l=1}^{k-1}C_{k-1}^{l-1}(1-F(s))^{k-j} F^{j-1}(s)=1-F^{k-1}(s).$$
Thus the initial probability is equal to
\begin{gather*}
\P(X_{(k,\{1, \ldots, k-1, s\})}/X_{(k-1,\{1, \ldots, k-1, s\})} < t)=
\frac{1}{t}(F^{k-1}(s)-1) -\textbf{1}( s\geq t)F^{k-1}(\frac{s}{t}).
\end{gather*}

Hence we get the final expression for the projection of the family of kernels $\Xi(\cdot,t):$
\begin{gather}\label{xi_k}
\xi_k(s;t) =\frac{1}{t}\left((1-\frac{1}{s})^{k-1}-\frac{1}{k}\right)-
\textbf{1}( s\geq t)\left((1-\frac{t}{s})^{k-1}-\frac{1}{k}\right).
\end{gather}

It is easy to show that $E(\xi_k (X; t))=0$. After some calculations we get that
the variance of this projection under $H_{0}$ is for any $t$
\begin{multline*}
\delta^2(t)=\frac{t+1}{(2k-1)t^2}+\frac{t-1}{k^2t^2}-
\sum_{i=0}^{k-1}\frac{(-1)^{j}2(k-1)!(k-1)!}{(k+j)!(k-j-1)!}t^{j-1}+\\+
(-1)^{k+1}\frac{2(k-1)!(k-1)!}{(2k-1)!}t^{k-2}F^{2k-1}(t)-\frac{2}{k^2t}F^k(t).
\end{multline*}

\subsection{ Kolmogorov-type statistic $D_n^{(3)}$}

In the case  $k=3$ the projection of the family of kernels $\Xi_3 (X,Y, Z;t),$ namely
$\xi_3 (s;t):=E(\Xi_3 (X, Y, Z; t)\mid X=s)$ is equal to:
\begin{equation}\label{xi_3}
\xi_3 (s;t)=  \frac{1}{t}(\frac{1}{s^2}-\frac{2}{s}+\frac23)- \textbf{1}\{s \geq t\}( \frac{t^2}{s^2}-\frac{2t}{s}+\frac23).
\end{equation}

Now we calculate the variances of these projections  $\delta_3^2(t)$ under $H_{0}.$ Elementary calculations show that
\begin{equation*}
\delta_3^2(t)=\frac{1}{45t^4}(4t^3+4t^2-15t+7).
\end{equation*}

Hence our family of kernels $\Xi_3 (X,Y,Z;t)$  is non-degenerate in the sense of \cite{Niki10} and
\begin{equation*}
\delta_3^2=\sup_{ t\geq 1} \delta_3^2(t)=0.03477.
\end{equation*}
This value will be important in the sequel when calculating the large deviation asymptotics.

\begin{figure}[h!]
\begin{center}
\includegraphics[scale=0.35]{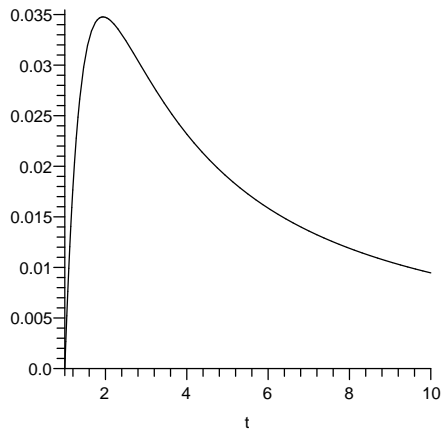}\caption{Plot of the function $\delta_3^2(t).$ }
\end{center}
\end{figure}

The limiting distribution of the statistic $D_n^{(3)}$ is unknown. Using the methods of \cite{Silv}, one can show that the
$U$-empirical process
$$\eta_n(t) =\sqrt{n} \left(H_n(t) - F_n(t)\right), \ t\geq 1,
$$
weakly converges in $D(1,\infty)$ as $n \to \infty$ to certain centered Gaussian
process $\eta(t)$ with calculable covariance. Then the sequence of statistics
$\sqrt{n} D_n^{(3)}$ converges in distribution to the rv   $\sup_{t\geq 1} |\eta(t)|$ but currently it is impossible to find explicitly its distribution.
Hence it is reasonable to determine the critical values for statistics  $D_n ^{(3)}$ by simulation.

Table \ref{fig: tab4} shows the critical values of the null distribution of $D_n^{(3)}$ for significance
levels $\alpha = 0.1, 0.05, 0.01$ and specific sample sizes $n.$ Each entry is obtained by using the Monte-Carlo simulation
methods with 10,000 replications.
\begin{table}[htbp]
\centering
\bigskip
\caption{Critical values for the statistic $D_n^{(3)}$ }
\bigskip
\centering
\begin{tabular}{c|ccc}

$n$ &  0.1 &  0.05 &  0.01   \\
  \hline
  10 & 0.333  & 0.400 & 0.558    \\
  20 & 0.254  & 0.277 & 0.331    \\
   30& 0.222  & 0.242 & 0.279   \\
  40 & 0.207  & 0.226 & 0.256    \\
   50 & 0.196  & 0.213 & 0.240     \\
   100 & 0.167 & 0.181 & 0.206     \\
\end{tabular}
\label{fig: tab4}
\end{table}

\medskip

Now we obtain the logarithmic large deviation asymptotics of the
sequence of statistics  $D_n^{(3)}$ under $H_0.$
The family of kernels $\{\Xi_3(X, Y, Z; t), t\geq 0\}$ is not only centered but bounded. Using the results from \cite{Niki10} on
large deviations for the supremum  of non-degenerate $U$-statistics, we obtain the following result.
\begin{theorem}
For  $a>0$
$$
\lim_{n\to \infty} n^{-1} \ln P (  D_n^{(3)} >a) = - f_D^{(3)}(a),
$$
where the function  $f_D^{(3)}$ is continuous for sufficiently small $a>0,$ moreover $$
f_D^{(3)}(a) = (18 \delta_3^2)^{-1} a^2(1 + o(1)) \sim 1.5978\, a^2, \, \mbox{as}
\, \, a \to 0.
$$
\end{theorem}

\bigskip

\subsection{Local efficiency of  $D_n^{(3)}$}

To evaluate the efficiency, first consider again the first Ley-Paindaveine alternative  with the d.f.
$G_1(x,\theta),\theta \geq 0, x \geq 1$ given above. By the Glivenko-Cantelli theorem
for $U$-statistics \cite{Jan} the limit in probability under the alternative for statistics $D_n^{(3)}$ is equal to
\begin{gather*}
b_1(\theta):= \sup_{t\geq 1}|b_1(t,\theta)|=
\sup_{t\geq 1}
|P_{\theta}(X_{(3,3)}/X_{(2,3)}<t)-G(t,\theta)|.
\end{gather*}
It is not difficult to show that
$$b_1(t,\theta) \sim 3\theta  \int_{1}^{\infty} \xi_3(s; t)h_1(s)ds,
$$
where again $
h_1(s)=\frac{\partial}{\partial\theta}g_1(s,\theta)\mid _{\theta=0}$ and $\xi_3(s;t)$ is the projection defined above in \eqref{xi_3}.
Hence for the first Ley-Paindaveine alternative we have for $t\geq 1:$
\begin{gather*}
b_1(t,\theta) \sim \frac{t-1}{2t^2}\theta, \quad \theta \to 0.
\end{gather*}

\begin{figure}[h!]
\begin{center}
\includegraphics[scale=0.35]{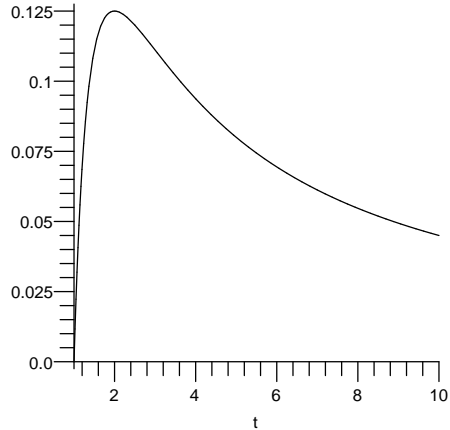}\caption{Plot of the function  $b_1(t,\theta), \mbox{ Ley-Paindaveine 1 alt.}$ }
\end{center}
\end{figure}
Thus $b_1(\theta)=\sup_{t\geq 1}|b_1(t,\theta)| \sim
 0.125\,\theta,$
and it follows that the local exact slope of the sequence of statistics
$D_n$ admits the representation:
$$c_1(\theta) \sim b^2_1(\theta)/(9\delta^2_3) \sim 0.0499\,\theta^2, \, \theta \to 0.$$
The Kullback-Leibler information in this case is  given by (\ref{K1}).
Hence the local Bahadur efficiency of our test is $e^B_1(D)= 0.599.$

Next we take the second Ley-Paindaveine distribution, where the calculations are similar, and
the local BE is equal to $0.689.$ In the case of the log-Weibull density
we find that the local BE is $0.467.$

We collect the values of local BE in the table \ref{fig: tab5}.
\begin{table}[!hhh]\centering
\caption{Local Bahadur efficiency for $D_n^{(3)}$ }
\bigskip
\begin{tabular}{|c|c|}
\hline
Alternative & Efficiency\\
\hline
Ley-Paindaveine 1 &  0.599  \\
Ley-Paindaveine 2 &   0.689 \\
log-Weibull&  0.467  \\
\hline
\end{tabular}
\label{fig: tab5}
\end{table}

\subsection{ Kolmogorov-type statistic $D_n^{(4)}$}

In the case  $k=4$ the projection of the family of kernels $\Xi_4 (X,Y, Z, W;t),$
is equal to:
\begin{gather*}\label{xi_4}
\xi_4 (s;t)=  \frac{1}{t}\left((1-\frac{1}{s})^3-\frac14\right)-
\textbf{1}\{s \geq t\}\left( -\left(\frac{t}{s}\right)^3+3\left(\frac{t}{s}\right)^2-\frac{3t}{s}+\frac34\right).
\end{gather*}

Therefore we get that the variances of these projections  $\delta_4^2(t)$ under $H_{0}$
\begin{equation*}
\delta_4^2(t)=\frac{1}{560t^5}(45t^4+45t^3-252t^2+224t-62).
\end{equation*}

Hence our family of kernels $\Xi_4 (X,Y,Z,W;t)$  is non-degenerate in the sense of \cite{Niki10} and
\begin{equation*}
\delta_4^2=\sup_{ t\geq 1} \delta_4^2(t)=0.0258.
\end{equation*}

\begin{figure}[h!]
\begin{center}
\includegraphics[scale=0.35]{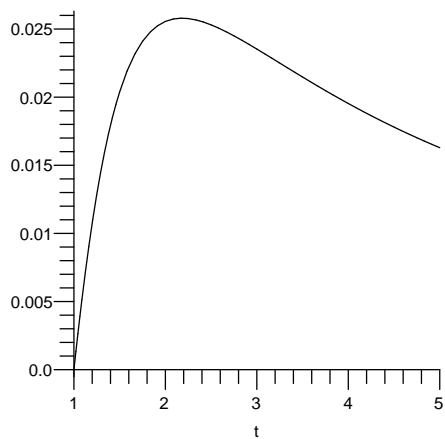}\caption{Plot of the function $\delta_4^2(t).$ }
\end{center}
\end{figure}

The limiting distribution of the statistic $D_n^{(4)}$ is unknown as in the previous section.
Using the Monte-Carlo methods, we again present the critical values of the null distribution for statistics
$D_n ^{(4)}$ for significance levels $\alpha = 0.1, 0.05, 0.01$ with 10,000 replications in the next table \ref{fig: tab6}.
\begin{table}[htbp]
\centering
\bigskip
\caption{Critical values for the statistic $D_n^{(4)}$ }
\bigskip
\centering
\begin{tabular}{c|ccc}

$n$ &  0.1 &  0.05 &  0.01   \\
  \hline
   10    & 0.400  & 0.433 & 0.600\\
  20 & 0.331  & 0.355 & 0.399\\
   30    & 0.304  & 0.328 & 0.362\\
  40     & 0.287  & 0.307 & 0.345\\
   50   & 0.276  & 0.295 & 0.328  \\
   100     & 0.244  & 0.260 & 0.285 \\
\end{tabular}
\label{fig: tab6}
\end{table}

\medskip

The logarithmic large deviation asymptotics of the
sequence of statistics  $D_n^{(4)}$ under $H_0$ is showed in the next theorem.
\begin{theorem}
For  $a>0$
$$
\lim_{n\to \infty} n^{-1} \ln P (  D_n^{(4)} >a) = - f_D^{(4)}(a),
$$
where the function  $f_D^{(4)}$ is continuous for sufficiently small $a>0,$ moreover $$
f_D^{(4)}(a) =(32 \, \delta_4^2)^{-1} a^2(1 + o(1)) \sim 1.211\, a^2, \, \mbox{as}
\, \, a \to 0.
$$
\end{theorem}

\bigskip

\subsection{Local efficiency of  $D_n^{(4)}$}

In table \ref{fig: tab7} we collect the calculated efficiencies for statistic $D_n^{(k)}$ joined with results from
table \ref{fig: tab5} and with the maximal values of efficiencies against our alternatives.
\begin{table}[!hhh]\centering
\bigskip
\caption{Comparative table of local efficiencies for statistic $D_n^{(k)}$ }
\bigskip
\begin{tabular}{|c|c|c|c|}
\hline
 & \multicolumn{3}{|c|}{Efficiency}\\
\cline{2-4}
\raisebox{1.5ex}[0cm][0cm]{Alternative}
 & $k=3$ & $k=4$ & $\max_k$ \\
\hline
Ley-Paindaveine 1 & 0.599 &  0.654 &  0.674 for $k=6$ \\
Ley-Paindaveine 2 & 0.689 &  0.767 &  0.790 for $k=5$\\
log-Weibull& 0.467 &  0.472  &  0.472 for $k=4$\\
\hline
\end{tabular}
\label{fig: tab7}
\end{table}

We observe that the efficiencies for the Kolmogorov-type test are lower than for the integral test. However,
it is the usual situation when testing goodness-of-fit \cite{Nik}, \cite{Rank}, \cite{Niki10}.

\bigskip
\section{Conditions of local asymptotic optimality}
\bigskip

In this section we are interested in conditions of local asymptotic optimality (LAO) in Bahadur sense
for both sequences of statistics $I_n^{(k)}$  and $D_n^{(k)}.$ This means to describe the local structure of the
alternatives for which the given statistic has maximal potential local efficiency so that the relation
$$
c_T(\theta) \sim 2 K(\theta),\, \theta \to 0,
$$
holds, see \cite{Nik}, \cite{NikTchir}.
Such alternatives form the domain of LAO for the given sequence of statistics.

Consider the functions
\begin{gather*}
H(x)=G^{'}_{\theta}(x,\theta)\mid
_{\theta=0},\quad
h(x)=g^{'}_{\theta} (x,\theta)\mid _{\theta=0}.
\end{gather*}
We will assume that the following regularity conditions are true, see also \cite{NikTchir}:
\begin{gather}
\int_{1}^{\infty} h^2(x)x \, dx <  \infty  \quad  \mbox{where} \quad  h(x)=H'(x), \, \label{PR1} \\
 \frac{\partial}{\partial\theta}\int_{1}^{\infty} g(x,\theta) \ln{x} \, dx \mid
_{\theta=0} \ = \ \int_{1}^{\infty} h(x)\ln{x} \, dx .  \label{PR2}
\end{gather}
Denote by $\cal G$ the class of densities  $ g(x,\theta)$  with d.f.'s $G(x,\theta),$ satisfying the regularity conditions (\ref{PR1}) - (\ref{PR2}).
We are going to deduce the LAO conditions in terms of the function $h(x).$

Recall that for alternative densities from  $\cal G$  the following asymptotics is valid:
\begin{gather*}
2K(\theta)\sim \theta^2 \left[ \int_{1}^{\infty} h^2(x)x\, dx -(\int_{1}^{\infty}  h(x)\ln{x} \, dx)^2\right],\, \theta \to 0.
\end{gather*}

\subsection{LAO conditions for $I_n^{(k)}$}

First consider the integral statistic  $I_n^{(k)}$  with the kernel
$\Psi_k(X_1, \ldots , X_{k+1})$ and its projection $\psi_k(x)$ from \eqref{psi_k}.
Let introduce the auxiliary function
$$
h_0(x) = h(x) - \frac{(\ln{x}-1)}{x^2}\int_1^\infty  h(u)\ln{u}\, du.
$$
Simple calculations show that
\begin{gather*}
\medskip
\int_{1}^{\infty} h^2(x)x^2dx -\left(\int_{1}^{\infty} h(x)\ln{x} \, dx\right)^2=\int_{1}^{\infty} h_0^2(x) x^2 dx,\\
\int_{1}^{\infty} \psi_k(x)h(x)dx = \int_{1}^{\infty} \psi_k(x)h_0(x)dx.
\end{gather*}

Hence the local asymptotic efficiency takes the form
\begin{multline*}
e^B (I_n^{(k)})= \lim_{\theta \to 0} b_I^2(\theta) / \left((k+1)^2\Delta_k^2 \cdot 2K(\theta)\right) =\\
= \left(\int_{1}^{\infty} \psi_k(x)h_0(x)dx\right)^2/\left(
\int_{1}^{\infty}\psi_k^2(x) \frac{dx}{x^2} \cdot   \int_{1}^{\infty} h_0^2(x)x^2 dx
 \right).
\end{multline*}

By Cauchy-Schwarz inequality
we obtain that  the expression in the right-hand side is equal to 1 iff
 $h_0(x)=C_1\psi_k(x)\frac{1}{x^2}$  for some constant $C_1>0,$
so that
\begin{equation}
h(x) =(C_1\psi_k(x)+ C_2 (\ln{x}-1))\frac{1}{x^2} \quad  \text{ for some constants } C_1>0 \text{ and } C_2.
\end{equation}
The set of distributions
for which the function $h(x)$ has such form generate the domain of LAO in the class $\cal G$.
The simplest examples of
such alternatives density  $g(x,\theta)$  for small $\theta > 0$ is given by the table \ref{fig: tab8}.
\begin{table}[!hhh]\centering
\medskip
\caption{Examples of LAO alternative density $g(x,\theta)$ for statistic $I_n^{(k)}$ }
\bigskip
\begin{tabular}{c|l}
 & Alternative density  $g(x, \theta)$ as $\theta \to +0, \, x \geq 1$\\[2mm]\hline
$k=3$ &  $g(x,\theta)=\frac{1}{x^2}(1+\theta\left( \frac{9}{8x^2}+\frac{3\ln{x}}{2x}-\frac1x-\frac14\right))$\\[2mm]\hline
$k=4$& $g(x,\theta)=\frac{1}{x^2}(1+\theta\left(\frac{12\ln{s}}{5s}-\frac{4}{5s^3}+\frac{18}{5s^2}-\frac{13}{5s}-\frac{3}{10}\right))$\\[2mm]\hline
\end{tabular}
\label{fig: tab8}
\end{table}

\subsection{LAO conditions for $D_n^{(k)}$}

Now let consider the Kolmogorov type statistic  $D_n^{(k)}$  with the family of kernels  $\Xi_k$ and their
projections $\xi_k(x;t)$ from \eqref{xi_k}. After simple calculations we get
\begin{gather*}
\int_{1}^{\infty} \xi_k(x; t)h(x)dx = \int_{1}^{\infty} \xi_k(x; t )h_0(x)dx, \quad  \forall t \in [1,\infty).
\end{gather*}

Hence the local efficiency takes the form
$$
e^B (D_n^{(k)})= \lim_{\theta \to 0} \left [ b_D^2(\theta)/ \sup_{t \geq 1}\left(
k^2 \delta_k^2(t)\right)\cdot 2K(\theta) \right ] = \frac{ \sup_{t \geq 1}\bigg(
\int_{1}^{\infty}  \xi_k(x;t)h_0(x)dx\bigg)^2 }{ \ \sup_{t \geq 1} \bigg(
\int_{1}^{\infty} \xi_k^2 (x;t) \frac{dx}{x^2} \cdot \int_{1}^{\infty} h_0^2(x)x^2 dx\bigg)} \leq 1.
$$

We can apply once again  the Cauchy-Schwarz inequality to the numerator in the last ratio. It follows that the
sequence of statistics $D_n$ is locally asymptotically optimal, and $e^B (D_n^{(k)})=1$
iff
\begin{equation*}\label{tt}
h(x)=(C_3\xi_k(x; t_0)+ C_4 (\ln{x}-1))\cdot \frac{1}{x^2} \quad \text{ for }
t_0= \arg \sup_{t \geq 1}\delta_k^2(t)\,
\end{equation*}
 and some constants $C_3>0 $ and $C_4.$

The distributions with such  $h(x)$ form the domain of LAO in the class  $\cal G$.
The simplest examples are given in the table \ref{fig: tab9}.
\begin{table}[!hhh]\centering
\bigskip
\caption{Examples of LAO alternative density $g(x,\theta)$ for statistic $D_n^{(k)}$ }
\bigskip
\begin{tabular}{c|l}
 & Alternative densities $g(x, \theta)$ as $\theta \to +0, \, x \geq 1$\\[2mm]\hline
$k=3$& $ g(x,\theta)=\frac{1}{x^2}\left(1+\theta \left(\frac{1}{t_1}(\frac{1}{x^2}-\frac{2}{x}+\frac23)- \textbf{1}\{x \geq t_1\}( (\frac{t_1}{x})^2-\frac{2t_1}{x}+\frac23)\right)\right)$\\[2mm]
& \quad \quad \quad \quad $t_1 = \arg \max_{t\geq 1}\left( \frac{1}{45t^4}(4t^3+4t^2-15t+7)\right)\approx 1.9395 $ \\[2mm]\hline
$k=4$ & $g(x,\theta)=\frac{1}{x^2}\left(1+\theta
\left(\frac{1}{t_2}((1-\frac{1}{x})^3-\frac14)- \textbf{1}\{x \geq t_2\}( -(\frac{t_2}{x})^3+3(\frac{t_2}{x})^2-\frac{3t_2}{x}+\frac34)\right)\right)$\\[2mm]
&\quad \quad \quad \quad $t_2= \arg \max_{t\geq 1}\left(
\frac{1}{560t^5}(45t^4+45t^3-252t^2+224t-62)\right)\approx 2.1810$\\[2mm]\hline
\end{tabular}
\label{fig: tab9}
\end{table}

\section{Conclusion}
We constructed two new tests for goodness-of-fit testing for Pareto distribution based on the new characterization
for the Pareto distribution. We describe their limit distribution and large deviations. The Bahadur efficiency
for some alternatives has been obtained and it turned
out reasonably high. Also we derived the conditions of local optimality for our tests.

\end{document}